\documentclass{article}
\usepackage{amsmath}
\usepackage{amsthm}
\usepackage{graphicx}
\usepackage{algorithm}
\usepackage{algorithmic}
\usepackage{amssymb}
\usepackage{geometry}
\usepackage{mathtools}
\usepackage{color}

\newtheorem{theorem}{Theorem}
\newtheorem{definition}{Definition}
\newtheorem{remark}{Remark}
\newtheorem{proposition}{Proposition}

\newtheorem{lemma}{Lemma}
\newtheorem{corollary}{Corollary}

\title{Generalized Fej\'er-Hermite-Hadamard type via generalized $(h-m)$-convexity on fractal sets and applications}
\author{Ohud Almutairi$^{1}$ and Adem Kiliçman$^{2}$  \\Department of Mathematics, University of Hafr Al-Batin, Hafr Al-Batin 31991,Sudia Arabia.\\ Universiti Putra Malaysia, Serdang 43400 UPM, Selangor, Malaysia\\ $^{1}$OhudbAlmutairi@gmail.com	and $^{2}$akilic@upm.edu.my}
\date{\today}

\begin{document}
	\maketitle
	
\noindent \textbf{Abstract}
In this article, we define a new class of convexity called generalized $(h-m)$-convexity, which generalizes $h$-convexity and $m$-convexity on fractal set $\mathbb{R}^{\alpha}$ $(0<\alpha\leq 1)$. Some properties of this new class are discussed. Using local fractional integrals and generalized $(h-m)$-convexity, we generalized Hermite-Hadamard (H-H) and Fej\'er-Hermite-Hadamard (Fej\'er-H-H) types inequalities. We also obtained a new result of the Fej\'er-H-H type for the function whose derivative in absolute value is the generalized $(h-m)$-convexity on fractal sets. As applications, we studied some new inequalities for random variables and numerical integrations.


\section{Introduction}

The H-H inequality plays essential roles in different areas of sciences, such as mathematics, physics and engineering (for example see \cite{Dragomir(2003),A2020,set(2021),DA(1998),Am(2019)}). This inequality provides estimates for the mean value of a continuous convex function. Therefore, the classical H-H inequality can be defined as follows.

\begin{theorem}
Let $\mathcal{G}:[\nu,\mu]\subseteq\mathbb{R}\to \mathbb{R}$ be a convex function on $[\nu,\mu]$ with $\nu<\mu$, then

\begin{equation}
\mathcal{G}\left(\frac{\nu+\mu}{2}\right) \leq \frac{1}{\mu-\nu}\int_{\nu}^{\nu} \mathcal{G}(x) d x \leq \frac{\mathcal{G}(\nu)+\mathcal{G}(\mu)}{2}\label{hh}
\end{equation}
holds.
\end{theorem}

Furthermore, the weighted generalization of inequality (\ref{hh}) is provided by Fej\'er \cite{Feje} as follows.

\begin{theorem}\label{fejetheorem}
Let $\mathcal{G}:[\nu,\mu]\subseteq\mathbb{R}\to \mathbb{R}$ be a convex function with $\nu<\mu$ and \(\mathcal{W}:[\nu, \mu] \rightarrow \mathbb{R}\) be an integrable, a non-negative and symmetric function with
respect to \((\nu+\mu) / 2\), then the inequality

\begin{equation}
\mathcal{G}\left(\frac{\nu+\mu}{2}\right)\int_{\nu}^{\mu} \mathcal{W}(x) d x \leq \int_{\nu}^{\mu} \mathcal{G}(x)\mathcal{W}(x) d x \leq \frac{\mathcal{G}(\nu)+\mathcal{G}(\mu)}{2}\int_{\nu}^{\mu} \mathcal{W}(x) d x\label{classicalfeje}
\end{equation}
holds.	
\end{theorem}
\begin{remark}
In Theorem \ref{fejetheorem}, choosing \(\mathcal{W}(x)=1\) reduces inequality (\ref{classicalfeje}) to inequality (\ref{hh}).
\end{remark}

Due to the enormous importance of inequalities (\ref{hh}) and (\ref{classicalfeje}), many generalizations of such inequalities involving a variant types of convexities have been investigated \cite{Dragomir2020,Ozcan(2019),setb2020,Al2020}. For more interesting results, one can consult the following references \cite{Almutairi2019,Qi2020,Sarikaya2013,A2019,Sarikaya(2008)}.

One important class of convexities is $m$-convexity which was studied by Toader \cite{Toader(1985)}, and is presented as follows.

\begin{definition}
	A function \(\mathcal{G}:[0, b] \rightarrow \mathbb{R}\) is called $m$-convex, where \(m \in[0,1]\), if the following inequality
	\begin{equation*}
	\mathcal{G}(\gamma \nu+m(1-\gamma) \mu) \leq \gamma \mathcal{G}(\nu)+m(1-\gamma) \mathcal{G}(\mu)
	\end{equation*}
	holds for every \(\nu, \mu \in[0, b]\) and \(\gamma \in[0,1]\).
\end{definition}
We say that \(\mathcal{G}\) is $m-$concave if \(-\mathcal{G}\) is $m$-convex.

In addition, Dragomir established H-H type inequalities for $m$-convexity \cite{Dragomir(2002)}, and we present the result as follows:

\begin{theorem}
Suppose that \(\mathcal{G}:[0, \infty) \rightarrow \mathbb{R}\) is $m$-convex, where $m \in(0,1]$ if
	\(\mathcal{G} \in L_{1}[\nu m, \mu]\) with \(0 \leq \nu<\mu<\infty,\) then we obtain
\begin{equation}
\frac{1}{m+1}\left[\int_{\nu}^{m \mu} \mathcal{G}(x) d x+\frac{m \mu-\nu}{\mu-m \nu} \int_{m \nu}^{\mu} \mathcal{G}(x) d x\right] \leq(m \mu-\nu) \frac{ \mathcal{G}(\nu)+ \mathcal{G}(\mu)}{2}.\label{mdragomir}
\end{equation}
\end{theorem}

Many studies have been conducted to extend and generalize $m$-convexity along with the related inequalities of H-H and Fej\'e-H-H types. For more studies, one can consult \cite{M(2020)} for exponentially $m$-convexity, \cite{Ozdemir(2016)} for $(h-m)$-convexity and \cite{Ozcan(2020)} for $(s-m)$-convexity. These types of convexities together with such inequalities can be extended to the fractional theory.

Fractional calculus is considered as an important area of study due to its wide-range applications to solve many real world problems. This can be seen in modelling by means
of fractals, control theory and random walk process \cite{Atangana2020,Da2020,K2020}. Following this, the theory of local fractional calculus have been used to generalize H-H and Fej\'e-H-H types inequalities for generalized $m$-convex function and other classes of convexities on fractal sets, such as Anastassiou et al. \cite{Anastassiou(2019)} for generalized strongly $m$-convex mapping. Other examples include the work of Abdeljwad et al. \cite{Abdeljawad(2020)} on generalized $(s, m) $-convex functions, Mo et al. \cite{Moetal(2014)} on generalized convex function, \"Ozcan et al. \cite{Ozcan(2019)} on $(\alpha-m)$-convexity, Du et al. \cite{Du(2019)} on generalized $m$-convex function and Luo et al. \cite{Luo(2020)} on generalized $h$-convex function.

Motivated by the above works, this study is therefore devoted to generalize local fractional inequalities of H-H and Fej\'er H-H types involving new class of convexity called generalized $(h-m)$-convex on fractal sets. We further present some properties of this new class. The relation between this class and earlier classes are presented here. New generalized inequalities of H-H and Fej\'er-H-H types for generalized $(h-m)$-convexity are also obtained. We  extended Fej\'e-H-H type inequalities for the class of mapping whose local fractional derivative in absolute value at a certain power is generalized $(h-m)$-convex. We applied our result to construct new inequalities for random variables and numerical integrations.

\section{Preliminaries}

This section introduces some known results involving the local fractional integrals. The theory of fractional set was proposed by Yang \cite{Yang(2012)}, and we present it as follows.

For \(0<\alpha \leq 1\), if \(\nu^{\alpha}, \mu^{\alpha}\) and \(\kappa^{\alpha}\) belong to the set \(\mathbb{R}^{\alpha}\) of real line numbers, then
\begin{itemize}
	\item [1.] \(\nu^{\alpha}+\mu^{\alpha}\) and \(\nu^{\alpha} \mu^{\alpha}\) belong to the set \(\mathbb{R}^{\alpha}\);
\item [2.]	\(\nu^{\alpha}+\mu^{\alpha}=\mu^{\alpha}+\nu^{\alpha}=(\nu+\mu)^{\alpha}=(\mu+\nu)^{\alpha}\);
\item [3.]\(\nu^{\alpha}+\left(\mu^{\alpha}+\kappa^{\alpha}\right)=(\nu+\mu)^{\alpha}+\kappa^{\alpha}\);
\item [4.]\(\nu^{\alpha} \mu^{\alpha}=\mu^{\alpha} \nu^{\alpha}=(\nu \mu)^{\alpha}=(\mu \nu)^{\alpha}\);
\item [5.]\(\nu^{\alpha}\left(\mu^{\alpha} \kappa^{\alpha}\right)=\left(\nu^{\alpha} \mu^{\alpha}\right) \kappa^{\alpha}\);
\item [6.]\(\nu^{\alpha}\left(\mu^{\alpha}+\kappa^{\alpha}\right)=\nu^{\alpha} \mu^{\alpha}+\nu^{\alpha} \kappa^{\alpha}\);
\item [7.]\(\nu^{\alpha}+0^{\alpha}=0^{\alpha}+\nu^{\alpha}=\nu^{\alpha}\) and \(\nu^{\alpha} 1^{\alpha}=1^{\alpha} \nu^{\alpha}=\nu^{\alpha}\).
\end{itemize}

In order to present the definition of local fractional integral on \(\mathbb{R}^{\alpha}\), the definition of the local fractional continuity is introduced as follows.

\begin{definition}
 A non-differentiable mapping \(\mathcal{G}: \mathbb{R} \rightarrow \mathbb{R}^{\alpha}, \zeta \rightarrow\)
	\(\mathcal{G}(\zeta)\) is named local fractional continuous at \(\zeta_{0},\) (or that \(\mathcal{G}(\zeta) \in C_{\alpha}(\nu, \mu)\)) if there exists
	\begin{equation*}
\left|\mathcal{G}(\zeta)-\mathcal{G}\left(\zeta_{0}\right)\right|<\varepsilon^{\alpha},
\end{equation*}	
with \(\left|\zeta-\zeta_{0}\right|<\varrho\) for any \(\varrho,\varepsilon>0\).

\end{definition}

Now, we give the definition of the local fractional  integral as follows.

\begin{definition}
 Let \(\mathcal{G}(x) \in C_{\alpha}[\nu, \mu] .\) Then the local fractional integral is
	defined by
	\begin{equation*}
	{ }_{\nu} \mathcal{I}_{\mu}^{\alpha} \mathcal{G}(x)=\frac{1}{\Gamma(\alpha+1)} \int_{\nu}^{\mu} \mathcal{G}(\lambda)(d \lambda)^{\alpha}=\frac{1}{\Gamma(\alpha+1)} \lim _{\Delta \lambda \rightarrow 0} \sum_{j=0}^{N-1} \mathcal{G}\left(\lambda_{j}\right)\left(\Delta \lambda_{j}\right)^{\alpha},
	\end{equation*}
	with \(\Delta \lambda_{j}=\lambda_{j+1}-\lambda_{j}\) and \(\Delta \lambda=\max \left\{\Delta \lambda_{1}, \Delta \lambda_{2}, \ldots, \Delta \lambda_{N-1}\right\},\) where \(\left[\lambda_{j}, \lambda_{j+1}\right], j=\)
	\(0, \ldots, N-1\) and \(\nu=\lambda_{0}<\lambda_{1}<\cdots<\lambda_{N-1}<\lambda_{N}=\mu\) is a partition of \([\nu, \mu]\).
\end{definition}
Here, it follows that \({ }_{\nu} \mathcal{I}_{\mu}^{\alpha} \mathcal{G}(x)=0\) if \(\nu=\mu\) and \({ }_{\nu} \mathcal{I}_{\mu}^{\alpha} \mathcal{G}(x)=-{ }_{\mu} I_{\nu}^{\alpha} \mathcal{G}(x)\) if \(\nu<\mu.\) If
for any \(x \in[\nu, \mu],\) there exists \({ }_{\nu} \mathcal{I}_{x}^{\alpha} \mathcal{G}(x),\) then we denote it by \(\mathcal{G}(x) \in \mathcal{I}_{x}^{\alpha}[\nu, \mu].\)


Yang \cite{Yang(2012)} established generalized H\"older's inequality by considering the local fractional integral.

\begin{lemma}
 If the functions \(\mathcal{G}, \mathcal{M} \in C_{\alpha}[\nu, \mu], \eta,\)
	\(\sigma>1\) where \(\frac{1}{\eta}+\frac{1}{\sigma}=1,\) then we get
	\begin{equation*}
	\begin{array}{c}\frac{1}{\Gamma(1+\alpha)} \int_{\nu}^{\mu}|\mathcal{G}(\rho) \mathcal{M}(\rho)|(\mathrm{d} \rho)^{\alpha} \leq\left(\frac{1}{\Gamma(\alpha+1)} \int_{\nu}^{\mu}|\mathcal{G}(\rho)|^{\eta}(\mathrm{d} \rho)^{\alpha}\right)^{\frac{1}{\eta}} \\ \left(\frac{1}{\Gamma(\alpha+1)} \int_{\nu}^{\mu}|\mathcal{M}(\rho)|^{\sigma}(\mathrm{d} \rho)^{\alpha}\right)^{\frac{1}{\sigma}}.\end{array}
	\end{equation*}
\end{lemma}

In \cite{Du(2019)}, Du et al. established the class of generalized $m$-convex functions on fractal set \(\mathbb{R}^{\alpha}\) together with integral inequalities of H-H type.

\begin{definition}
Let \(\mathcal{G}:\left[0, b\right] \rightarrow \mathbb{R}^{\alpha}\), with \(b>0\). For any \(\nu, \mu \in\left[0, b\right]\), \(\gamma \in[0,1]\) and \(m \in(0,1]\), if the following inequality
	\begin{equation*}
	\mathcal{G}(\gamma \nu+m(1-\gamma) \mu) \leq \gamma^{\alpha} \mathcal{G}(\nu)+m^{\alpha}(1-\gamma)^{\alpha} \mathcal{G}(\nu)
	\end{equation*}
	holds, then $\mathcal{G}$ is said to be generalized $m$-convex.
\end{definition}

\begin{theorem}
Suppose that \(\mathcal{G}:[0, \infty) \rightarrow \mathbb{R}^{\alpha}\) is a generalized
	\(m\)-convex mapping, where \(m \in(0,1]\) and \(0 \leq \nu<\mu.\) If \(\mathcal{G}(x) \in\)
	${ }_{\nu} I_{\mu}^{\alpha}[\nu, \mu]$, then the following
	\begin{equation*}
	\begin{aligned} \frac{\mathcal{G}\left(\frac{\nu+\mu}{2}\right)}{\Gamma(1+\alpha)} \leq & \frac{\nu \mathcal{I}_{\mu}^{(\alpha)}\left[\mathcal{G}(x)+m^{\alpha} \mathcal{G}\left(\frac{x}{m}\right)\right]}{(2(\mu-\nu))^{\alpha}} \\ \leq &\left(\frac{1}{4}\right)^{\alpha} \frac{\Gamma(1+\alpha)}{\Gamma(1+2 \alpha)}[\mathcal{G}(\nu)+\mathcal{G}(\mu)\\ &+2^{\alpha} m^{\alpha}\left(\mathcal{G}\left(\frac{\nu}{m}\right)+\mathcal{G}\left(\frac{\mu}{m}\right)\right) \\ &\left.+m^{2 \alpha}\left(\mathcal{G}\left(\frac{\nu}{m^{2}}\right)+\left(\frac{\mu}{m^{2}}\right)\right)\right] \end{aligned}
	\end{equation*}
	holds, for all \(x \in[\nu, \mu]\).
\end{theorem}

The other class of generalized convex functions include $h$-convex mapping on fractal set \(\mathbb{R}^{\alpha}\). In \cite{Vivas(2016)}, the generalized $h$-convex mapping on fractal set was introduced, through which new inequalities of H-H type were studied. 

\begin{definition}
	 Let \(h: I \subseteq \mathbb{R} \rightarrow \mathbb{R}^{\alpha}\) be a non-negative mapping
	and \(h \not\equiv 0^{\alpha}.\) The mapping \(\mathcal{G}: J\subseteq \mathbb{R} \rightarrow \mathbb{R}^{\alpha}\) is said to be generalized
	\(h\)-convex if \(\mathcal{G}\) is non-negative and the inequality
	\begin{equation*}
	\mathcal{G}(\gamma \nu+(1-\gamma) \mu) \leq h(\gamma) \mathcal{G}(\nu)+h(1-\gamma) \mathcal{G}(\mu)
	\end{equation*}
	holds, for all \(\nu, \mu \in J\) and \(\gamma \in(0,1)\).
\end{definition}

\begin{theorem}\label{vivastheorem}
	Suppose that \(h: I \subseteq \mathbb{R} \rightarrow \mathbb{R}^{\alpha}\) is a positive integral function with \(h \not\equiv 0^{\alpha}.\) Let \(\mathcal{G}: J \subseteq \mathbb{R} \rightarrow \mathbb{R}^{\alpha}\) be an \(h\)-convex, positive and integrable function, \(\mu, \nu \in J\) with \(\nu<\mu\), then the following inequality
\begin{equation*}
\begin{array}{l}\frac{1^{\alpha}}{\left(1^{\alpha}-(-1)^{\alpha}\right) h(1 / 2) \Gamma(\alpha+1)} \mathcal{G}\left(\frac{\nu+\mu}{2}\right) \\ \quad \leq \frac{1^{\alpha}}{(\mu-\nu)^{\alpha}} { }_{\nu} \mathcal{I}_{\mu}^{(\alpha)} \mathcal{G} \leq\left(\mathcal{G}(\mu)-(-1)^{\alpha} \mathcal{G}(\nu)\right){ }_{0} \mathcal{I}_{1}^{(\alpha)} h\end{array}
\end{equation*}
holds.
\end{theorem}

Luo et. al \cite{Luo(2020)} studied new inequalities of Fej\'er-H-H type via generalized
$h$-convexity on fractal sets. This was achieved using the following lemma.

\begin{lemma}\label{luolemma}
Suppose that \(\mathcal{G}: J \subseteq \mathbb{R} \rightarrow \mathbb{R}^{\alpha}\) is local continuous on the interior of J, \(J^{\circ}\). Let \(\mathcal{W}:[\nu, \mu] \rightarrow \mathbb{R}^{\alpha}, \mathcal{W} \geq 0^{\alpha}\) be symmetric to \(\frac{\nu+\mu}{2}\) and local continuous. If \(\mathcal{G}^{(\alpha)} \in C_{\alpha}[\nu, \mu]\) for \(\nu, \mu \in J\) with \(\nu<\mu,\) then the
following identity	
	\begin{equation*}
	\begin{array}{l}\frac{\mathcal{G}(\nu)+\mathcal{G}(\mu)}{2^{\alpha}} {}_{\nu} \mathcal{I}_{\mu}^{(\alpha)} \mathcal{W}(x)-{ }_{\nu} \mathcal{I}_{\mu}^{(\alpha)} \mathcal{W}(x) \mathcal{G}(x) \\ =\left(\frac{\mu-\nu}{4}\right)^{\alpha} \frac{1}{\Gamma(1+\alpha)} \int_{0}^{1}\left[\frac{1}{\Gamma(1+\alpha)} \int_{m(\gamma)}^{n(\gamma)} \mathcal{W}(x)(dx)^{\alpha}\right] \\ \quad\left(\mathcal{G}^{(\alpha)}(n(\gamma))-\mathcal{G}^{(\alpha)}(m(\gamma))\right)(d \gamma)^{\alpha}\end{array}
	\end{equation*}
	holds, where \(m(\gamma)=\gamma \nu+(1-\gamma) \frac{\nu+\mu}{2}, n(\gamma)=\gamma \mu+(1-\gamma) \frac{\nu+\mu}{2}\) and \(\gamma \in[0,1]\).
\end{lemma}

\begin{theorem}\label{luo}
 Suppose that \(h:I \subseteq \mathbb{R} \rightarrow \mathbb{R}^{\alpha}\) is a positive mapping, and
 \(\mathcal{W}:[\nu, \mu] \rightarrow \mathbb{R}^{\alpha}, \mathcal{W} \geq 0^{\alpha}\) is symmetric to \(\frac{\nu+\mu}{2}.\) If \(\mathcal{G}:\)\([\nu, \mu] \rightarrow \mathbb{R}^{\alpha}\) is generalized $h$-convex and \(\mathcal{G}(x), \mathcal{W}(x) \in \mathcal{I}_{x}^{\alpha}[\nu, \mu],\) with \(h\left(\frac{1}{2}\right) \neq 0^{\alpha}\), then we have
	\begin{equation}
	\begin{array}{l}\frac{\mathcal{G}\left(\frac{\nu+\mu}{2}\right){ }_{\nu} \mathcal{I}_{\mu}^{(\alpha)} \mathcal{W}(x)}{2^{\alpha} h\left(\frac{1}{2}\right)} \leq{ }_{\nu} \mathcal{I}_{\mu}^{(\alpha)} \mathcal{W}(x) \mathcal{G}(x) \\ \quad \leq \frac{\mathcal{G}(\nu)+\mathcal{G}(\mu)}{2^{\alpha}}{ }_{\nu} \mathcal{I}_{\mu}^{\alpha}\left[h\left(\frac{\mu-x}{\mu-\nu}\right)+h\left(\frac{x-\nu}{\mu-\nu}\right)\right] \mathcal{W}(x).\end{array}
	\end{equation}
\end{theorem}

\section{New definition and properties}
The concept of generalized $(h-m)$-convex mappings can be introduced as
follows.

\begin{definition}\label{defmh}
	Suppose that \((0,1)\subseteq I \subseteq \mathbb{R}\) and \(h: I \rightarrow \mathbb{R}^{\alpha}\) is a non-negative function. We say that \(\mathcal{G}:[0, b] \rightarrow \mathbb{R}^{\alpha}\) is generalized $(h-m)$-convex function, if \(\mathcal{G}\) is positive, then the following inequality
	\begin{equation}
	\mathcal{G}(\gamma \nu+m(1-\gamma) \mu) \leq h(\gamma) \mathcal{G}(x)+m^{\alpha} h(1-\gamma) \mathcal{G}(\mu)\label{mh}
	\end{equation}
holds, for all \(\nu, \mu\in[0, b], m \in[0,1]\) and \(\gamma \in[0,1]\).	
\end{definition}
If inequality (\ref{mh}) is reversed, then we say that \(\mathcal{G}\) is generalized $(h-m)$-concave on fractal set.
\begin{corollary}
Choosing \(\gamma=\frac{1}{2}\) in inequality (\ref{mh}) of Definition \ref{defmh}, we get Jensen-type \((h, m)\)-convex on fractal set as follows:
	\begin{equation*}
	\mathcal{G}\left(\frac{\nu+m \mu}{2}\right) \leq h\bigg(\frac{1}{2}\bigg)\left[\mathcal{G}(\nu)+m^{\alpha} \mathcal{G}(\mu)\right].
	\end{equation*}
\end{corollary}

Some particular cases of Definition \ref{defmh} are presented in the following remark.
\begin{remark}
	Consider  Definition \ref{defmh}.

	\begin{itemize}
		\item[i.] If $h(\gamma)=\gamma^{\alpha}$ and $m=1$, we get generalized convex function \cite{Moetal(2014)}.
			\item [ii.] If $h(\gamma)=\gamma^{s\alpha }$, we get generalized $(s-m)$ convex \cite{Abdeljawad(2020)}.
			
		\item [iii.] If $h(\gamma)=\gamma^{\alpha}$, we have generalized $m$-convexity on fractal sets \cite{Du(2019)}.
			\item [iv.] If $m=1$, we get the class of generalized $h$-convexity on fractal sets \cite{Vivas(2016)}.
			\item [v] If  $\alpha=1$ , we obtain the class of $(h-m)$-convexity \cite{Ozdemir(2016)}. 
			\item[vi.] If $\alpha=m=1$, we get the class of $h$-convexity \cite{Varosanec(2007)}.
\end{itemize}
\end{remark}

\begin{proposition}\label{prop}
	Suppose that $h_{1}$ and $h_{2}$ are positive functions defined on \(I \subseteq \mathbb{R}\), such that
	\begin{equation*}
	h_{2}(\gamma) \leq h_{1}(\gamma)
	\end{equation*}
	for \(\gamma \in(0,1) .\) If \(\mathcal{G}\) is generalized $(h_{2}-m)$-convex on fractal sets, then \(\mathcal{G}\) is generalized $(h_{1}-m)$-convex on fractal sets.
\end{proposition}

\begin{proof}
	Let \(\mathcal{G}\) be a generalized $(h_{2}-m)$-convex on fractal sets, then we obtain the following inequality
	\begin{equation*}
	\begin{aligned} \mathcal{G}(\gamma \nu+m(1-\gamma) \mu) & \leq h_{2}(\gamma) \mathcal{G}(x)+m h_{2}(1-\gamma) \mathcal{G}(\mu) \\ & \leq h_{1}(\gamma) \mathcal{G}(x)+m h_{1}(1-\gamma) \mathcal{G}(\mu), \end{aligned}
	\end{equation*}
for all \(\nu, \mu \in[0, b]\) and \(\gamma \in(0,1)\).
This completes the proof of Proposition \ref{prop}.
\end{proof}

\begin{proposition}
 If \(\mathcal{G}, \mathcal{M}\) are generalized $(h-m)$-convex functions on fractal sets and \(\lambda>0,\) then we have the following properties:
 \begin{itemize}
 	\item [i.] $\mathcal{G}+\mathcal{M}$ is generalized $(h-m)$-convex function on fractal sets.
 	\item[ii.] $\lambda \mathcal{G}$ is generalized $(h-m)$-convex function on fractal sets. 
 \end{itemize}
\end{proposition}

\begin{proof}
\item [i.] Using the definition of generalized $(h-m)$-convex functions on fractal sets, we have
\begin{equation}
\mathcal{G}(\gamma \nu+m(1-\gamma) \mu) \leq h(\gamma) \mathcal{G}(\nu)+m^{\alpha} h(1-\gamma) \mathcal{G}(\mu)\label{pf}
\end{equation}
and	
\begin{equation}
\mathcal{M}(\gamma \nu+m(1-\gamma) \mu) \leq h(\gamma) \mathcal{M}(\nu)+m^{\alpha} h(1-\gamma) \mathcal{M}(\mu),\label{ps}
\end{equation}
for all \(\nu, \mu \in[0, b], m \in[0,1]\) and \(\gamma \in(0,1).\)
Combining the inequalities (\ref{pf}) and (\ref{ps}), we get
\begin{equation*}
(\mathcal{G}+\mathcal{M})(\gamma \nu+m(1-\gamma) \mu) \leq h(\gamma)(\mathcal{G}+\mathcal{M})(\nu)+m^{\alpha} h(1-\gamma)(\mathcal{G}+\mathcal{M})(\mu).
\end{equation*}

\item [ii.]	The proof follows immediately from property (\ref{mh}) of Definition \ref{defmh}. 
\end{proof}

\section{Generalized inequalities of H-H type via generalized $(h-m)$-convexity on fractal sets}
Local fractional inequalities of H-H type via generalized $(h-m)$-convexity can be presented as follows.
\begin{theorem}\label{theoremhh}
Suppose that \(h:(0,1) \subset I\rightarrow \mathbb{R}^{\alpha}\) is a non-negative integrable function such that \(h \not\equiv  0\). Let \(\mathcal{G}: J \rightarrow \mathbb{R}^{\alpha}\) be a
positive, generalized $(h-m)$-convex and integrable function, with \(\nu, \mu \in J\) and \(\nu<\mu\). Then the following inequality

\begin{flalign}
&\frac{1^{\alpha}}{\Gamma(1+\alpha)}\mathcal{G}\left(\frac{\nu+\mu}{2}\right) \leq  h\left(\frac{1}{2}\right)\frac{\nu \mathcal{I}_{\mu}^{(\alpha)}[\mathcal{G}(x)+m^{\alpha} \mathcal{G}(\frac{x}{m})] }{(\mu-\nu)^{\alpha}}\nonumber\\& \leq h\left(\frac{1}{2}\right)\frac{1^{\alpha}}{\Gamma(\alpha+1)}\bigg[\mathcal{G}(\nu)+m^{2 \alpha}\mathcal{G}\left(\frac{\mu}{m^{2}}\right) - m^{\alpha}(-1)^{\alpha}\left(\mathcal{G}\left(\frac{\mu}{m}\right)+\mathcal{G}\left(\frac{\nu}{m}\right)\right)\bigg]{}_{0} \mathcal{I}_{1}^{(\alpha)} h\label{hhhmtheorem}
\end{flalign}

holds.
\end{theorem}

\begin{proof}

	Since \(\mathcal{G}\) is generalized \(h-m\)-convex on fractal sets, we get
	\begin{equation}
	\mathcal{G}\left(\frac{y+x}{2}\right) \leq h\left(\frac{1}{2}\right)\left[\mathcal{G}(y)+m^{\alpha} \mathcal{G}\left(\frac{x}{m}\right)\right].\label{fprofth}
	\end{equation}

	Substituting $y=\gamma \nu+(1-\gamma) \mu$ and $x=(1-\gamma) \nu+\gamma \mu$ in inequality (\ref{fprofth}), we have
	\begin{equation}
\mathcal{G}\left(\frac{\nu+\mu}{2}\right) \leq h\left(\frac{1}{2}\right)\left[\mathcal{G}(\gamma \nu+(1-\gamma) \mu)+m^{\alpha} \mathcal{G}\left((1-\gamma) \frac{\nu}{m}+\gamma\left(\frac{\mu}{m}\right)\right)\right].\label{sprofth}
 \end{equation}

Integrating inequality (\ref{sprofth}) corresponding to \(\gamma\) over \([0,1],\) we have

\begin{flalign}
\int_{0}^{1}& \mathcal{G}\left(\frac{\nu+\mu}{2}\right)(d \gamma)^{\alpha} \leq \bigg[ h\bigg(\frac{1}{2}\bigg) \int_{0}^{1} \mathcal{G}(\gamma \nu+(1-\gamma) \mu)(d \gamma)^{\alpha}\nonumber\\&+h\bigg(\frac{1}{2}\bigg) m^{\alpha}\int_{0}^{1} \mathcal{G}\left((1-\gamma) \frac{\nu}{m}+\gamma \frac{\mu}{m}\right)( d\gamma)^{\alpha}\bigg].\label{proff}
 \end{flalign}

It is easy to see that

 \begin{equation}
 \int_{0}^{1} \mathcal{G}(\gamma \nu+(1-\gamma) \mu)(d \gamma)^{\alpha}=\frac{-(-1)^{\alpha}}{(\mu-\nu)^{\alpha}} \int_{0}^{1} \mathcal{G}(x)(d x)^{\alpha},\label{profs}
\end{equation}

 and 
 \begin{equation}
\int_{0}^{1} \mathcal{G}\left((1-\gamma) \frac{\nu}{m}+\gamma \frac{\mu}{m}\right)(d \gamma)^{\alpha} = \frac{-(-1)^{\alpha}}{(\mu-\nu)^{\alpha}} \int_{0}^{1} \mathcal{G}\bigg(\frac{x}{m}\bigg)(d x)^{\alpha}.\label{proft}
 \end{equation}

In view of identities (\ref{profs}), (\ref{proft}) and inequality (\ref{proff}), we get 
\begin{equation*}
\frac{1^{\alpha}}{\Gamma(1+\alpha)}\mathcal{G}\left(\frac{\nu+\mu}{2}\right) \leq  h\left(\frac{1}{2}\right)\frac{\nu \mathcal{I}_{\mu}^{(\alpha)}\left[\mathcal{G}(x)+m^{\alpha} \mathcal{G}\left(\frac{x}{m}\right)\right]}{(\mu-\nu)^{\alpha}}.
\end{equation*}

Then the first part of inequality (\ref{hhhmtheorem}) is proved.

To show the second part of inequality (\ref{hhhmtheorem}), we used the generalized $(h-m)$-convexity on fractal sets of $\mathcal{G}$. Thus,

\begin{flalign}
&h\left(\frac{1}{2}\right)\left[\mathcal{G}(\gamma \nu+(1-\gamma) \mu)+m^{\alpha} \mathcal{G}\bigg((1-\gamma) \frac{\nu}{m}+\gamma\left(\frac{\mu}{m}\right)\bigg)\right]\nonumber\\& \leq h\left(\frac{1}{2}\right)\left[h(\gamma) \mathcal{G}(\nu)+m^{\alpha}h(1-\gamma) \mathcal{G}\left(\frac{\mu}{m}\right) +m^{\alpha}h(1-\gamma) \mathcal{G}\left(\frac{\nu}{m}\right)+m^{2 \alpha} h(\gamma) \mathcal{G}\left(\frac{\mu}{m^{2}}\right)\right].\label{midillproof}
\end{flalign}

Integrating the inequality (\ref{midillproof}) with
respect to \(\gamma \in[0,1],\) we obtain

\begin{flalign*}
&h\left(\frac{1}{2}\right)\frac{\nu \mathcal{I}_{\mu}^{(\alpha)}[\mathcal{G}(x)+m^{\alpha} \mathcal{G}(\frac{x}{m})] }{(\mu-\nu)^{\alpha}} \\&\leq\Gamma(\alpha+1) h\bigg(\frac{1}{2}\bigg)\bigg[\mathcal{G}(\nu)\int_{0}^{1}h(\gamma)d(\gamma)^{\alpha} \\&+m^{\alpha} \int_{0}^{1}h(1-\gamma)d(\gamma)^{\alpha} \mathcal{G}(\frac{\mu}{m}) +m^{\alpha}\int_{0}^{1}h(1-\gamma)d(\gamma)^{\alpha} \mathcal{G}\bigg(\frac{\nu}{m}\bigg)\\&+m^{2 \alpha}\int_{0}^{1} h(\gamma)d(\gamma)^{\alpha} \mathcal{G}\bigg(\frac{\mu}{m^{2}}\bigg)\bigg].
\end{flalign*}

Thus,

\begin{flalign*}
h\left(\frac{1}{2}\right)\frac{\nu \mathcal{I}_{\mu}^{(\alpha)}[\mathcal{G}(x)+m^{\alpha} \mathcal{G}(\frac{x}{m})] }{(\mu-\nu)^{\alpha}} &\leq h\left(\frac{1}{2}\right)\frac{1^{\alpha}}{\Gamma(\alpha+1)}\bigg[\bigg[\mathcal{G}(\nu)+m^{2 \alpha}\mathcal{G}\left(\frac{\mu}{m^{2}}\right)\bigg]\int_{0}^{1}h(\gamma) d(\gamma)^{\alpha} \\&+ \left[-m^{\alpha}(-1)^{\alpha}\left(\mathcal{G}\left(\frac{\mu}{m}\right)+\mathcal{G}\left(\frac{\nu}{m}\right)\right)\right]\int_{0}^{1}h(\gamma) d(\gamma)^{\alpha}\bigg].
\end{flalign*}

Therefore,

\begin{flalign*}
h\left(\frac{1}{2}\right)\frac{\nu \mathcal{I}_{\mu}^{(\alpha)}[\mathcal{G}(x)+m^{\alpha} \mathcal{G}(\frac{x}{m})] }{(\mu-\nu)^{\alpha}} \leq h\left(\frac{1}{2}\right)\frac{1^{\alpha}}{\Gamma(1+\alpha)}\bigg[\mathcal{G}(\nu)+m^{2 \alpha}\mathcal{G}\left(\frac{\mu}{m^{2}}\right) - m^{\alpha}(-1)^{\alpha}\left(\mathcal{G}\left(\frac{\mu}{m}\right)+\mathcal{G}\left(\frac{\nu}{m}\right)\right)\bigg]{}_{0} \mathcal{I}_{1}^{(\alpha)} h.
\end{flalign*}
\end{proof}

\begin{corollary}
Choosing $m=1$ and $h(\alpha)=\alpha$ in inequality (\ref{hhhmtheorem}) of Theorem \ref{theoremhh}, we get Theorem 14
studied by Mo et al. \cite{Moetal(2014)}. Taking $\alpha=1$ in inequality (\ref{hhhmtheorem}) of Theorem \ref{theoremhh}, we obtain Theorem 9 given by \"Ozdemir et al. \cite{Ozdemir(2016)}. Choosing $\alpha=1$ and $h(\gamma)=\gamma$ in inequality (\ref{hhhmtheorem}) of Theorem \ref{theoremhh}, we obtain Theorem 4 given by Dragomir \cite{Dragomir(2002)}. Taking $m=1$ in Theorem \ref{theoremhh}, we get Theorem \ref{vivastheorem} established by Vivas et al. \cite{Vivas(2016)}. Taking $h(\gamma)=\gamma$ of Theorem \ref{theoremhh}, we have Theorem 3.1 given by Du et al. \cite{Du(2019)}.
\end{corollary}

\begin{theorem}
Let \(\mathcal{G}:J\rightarrow \mathbb{R}\) be generalized $(h-m)$ convex on fractal sets  with
	\(\gamma \in[0,1]\) and \(m \in(0,1]\). If \(0 \leq \nu<\mu<\infty\) and \(\mathcal{G} \in L_{1}[m \nu, \mu],\) then we have
	\begin{flalign}
	\frac{1}{m^{\alpha}+1}&\bigg[\frac{1}{m \mu-\nu} \int_{\nu}^{\mu } \mathcal{G}(x) d x+\frac{1}{\mu-m \nu} \int_{m \nu}^{\mu} \mathcal{G}(x) d x\bigg]\nonumber\\& \leq (\mathcal{G}(\nu)+\mathcal{G}(\mu))\bigg[\int_{0}^{1} h(\gamma) d \gamma+\int_{0}^{1} h(1-\gamma) d \gamma\bigg].\label{parhtheor}
	\end{flalign}
\end{theorem}

\begin{proof}
	From the definition of generalized $(h-m)$-convex on fractal sets, we can write
	\begin{equation*}
	\mathcal{G}(\gamma \nu+m(1-\gamma) \mu) \leq h(\gamma) \mathcal{G}(\nu)+m^{\alpha} h(1-\gamma) \mathcal{G}(\mu),
	\end{equation*}
	\begin{equation*}
	\mathcal{G}((1-\gamma) \nu+m \gamma \mu) \leq h(1-\gamma) \mathcal{G}(\nu)+m^{\alpha} h(\gamma) \mathcal{G}(\mu),
	\end{equation*}
	\begin{equation*}
	\mathcal{G}(\gamma \mu+(1-\gamma) m \nu) \leq h(\gamma) \mathcal{G}(\mu)+m^{\alpha}h(1-\gamma) \mathcal{G}(\nu),
	\end{equation*}
	and
	\begin{equation*}
	\mathcal{G}((1-\gamma) \mu+\gamma m \nu) \leq h(1-\gamma) \mathcal{G}(\mu)+m^{\alpha} h(\gamma) \mathcal{G}(\nu).
	\end{equation*}
Combining the above inequalities, we get
	\begin{equation}
	\begin{array}{l}\mathcal{G}(\gamma \nu+m(1-\gamma) \mu)+\mathcal{G}((1-\gamma) \nu+m \gamma \mu) \\ +\mathcal{G}(\gamma \mu+(1-\gamma) m \nu)+\mathcal{G}((1-\gamma) \mu+\gamma m \nu) \\ \leq [\mathcal{G}(\nu)+\mathcal{G}(\mu)](m^{\alpha}+1)[h(\gamma)+h(1-\gamma)].\end{array} \label{midllineq}
	\end{equation}
 Integrating inequality (\ref{midllineq}) on $[0,1]$ with respect to \(\gamma,\) we obtain
	
\begin{equation*}
\begin{array}{l}\int_{0}^{1} \mathcal{G}(\gamma \nu+m(1-\gamma) \mu) d \gamma+\int_{0}^{1} \mathcal{G}((1-\gamma) \nu+m \gamma \mu) d \gamma \\ \quad+\int_{0}^{1} \mathcal{G}(\gamma \mu+m(1-\gamma) \nu) d \gamma+\int_{0}^{1} \mathcal{G}((1-\gamma) \mu+m \gamma \nu) d \gamma \\ \leq(\mathcal{G}(\nu)+\mathcal{G}(\mu))(m^{\alpha}+1)\left[\int_{0}^{1} h(\gamma) d \gamma+\int_{0}^{1} h(1-\gamma) d \gamma\right],\end{array}
\end{equation*}	
where
\begin{equation*}
\int_{0}^{1} \mathcal{G}(\gamma \nu+m(1-\gamma) \mu) d \gamma=\int_{0}^{1} \mathcal{G}((1-\gamma) \nu+m \gamma \mu) d \gamma=\frac{1}{m \mu-\nu} \int_{\nu}^{m \mu} \mathcal{G}(x) d x
\end{equation*}

and
\begin{equation*}
\int_{0}^{1} \mathcal{G}(\gamma \mu+m(1-\gamma) \nu) d \gamma=\int_{0}^{1} \mathcal{G}((1-\gamma) \mu+m \gamma \nu) d \gamma=\frac{1}{\mu-m \nu} \int_{m \nu}^{\mu} \mathcal{G}(x) d x.
\end{equation*}
	
\end{proof}

\begin{corollary}
 Choosing \(h(\gamma)=1\) in inequality (\ref{parhtheor}), we obtain
 \begin{equation*}
 \frac{1}{m^{\alpha}+1}\left[\frac{1}{m \mu-\nu} \int_{\nu}^{m \mu} \mathcal{G}(x) d x+\frac{1}{\mu-m \nu} \int_{m \nu}^{\mu} \mathcal{G}(x) d x\right] \leq \mathcal{G}(\nu)+\mathcal{G}(\mu).
 \end{equation*}
\end{corollary}

\begin{remark}
Choosing $\alpha=1$ and $h(\gamma)=\gamma$ in inequality (\ref{parhtheor}), we obtain inequality (\ref{mdragomir}) given by Dragomir \cite{Dragomir(2002)}. Taking $\alpha=1$ in inequality (\ref{parhtheor}), we get inequality (2.8) of Theorem 10 established by \"Ozdemir et al. \cite{Ozdemir(2016)}. Choosing $\alpha=m=1$ and $h(\gamma)=\gamma$ in inequality (\ref{parhtheor}), we obtain the right hand
side of H-H inequality (\ref{hh}). Choosing $\alpha=m=1$ and $h(\gamma)=\gamma^{s}$ in inequality (\ref{parhtheor}), we obtain the second part
of inequality (2.1) established by Dragomir and Fitzpatrick \cite{Dragomir(1999)}.
\end{remark}

\section{Fej\'er H-H type inequalities via generalized $(h-m)$ convexity}
In this section, the generalized $(h-m)$ convexity is used to present certain inequalities of Fej\'er H-H type.
\begin{theorem}\label{fejeresult}
 Let \(h: I \subseteq \mathbb{R} \rightarrow \mathbb{R}^{\alpha}\) be a non-negative mapping.
	Suppose that \(\mathcal{W}:[\nu, \mu] \rightarrow \mathbb{R}^{\alpha}, \mathcal{W} \geq 0^{\alpha}\) is symmetric to \(\frac{\nu+\mu}{2}\) and \(\mathcal{G}(x), \mathcal{W}(x) \in \mathcal{I}_{x}^{\alpha}[\nu, \mu],\) where $h\left(\frac{1}{2}\right)\not\equiv  0^{\alpha}$ and $m \in(0,1]$. If \(\mathcal{G}:\)
	\([\nu, \mu] \rightarrow \mathbb{R}^{\alpha}\) is generalized $(h-m)$-convex, then the following inequality
	\begin{footnotesize}
\begin{flalign}
\frac{\mathcal{G}\left(\frac{\nu+\mu}{2}\right) { }_{\nu} \mathcal{I}_{\mu}^{(\alpha)} \mathcal{W}(x)}{2^{\alpha} h\left(\frac{1}{2}\right)} &\leq { }_{\nu} \mathcal{I}_{\mu}^{(\alpha)} \frac{\mathcal{G}(x)+ m^{\alpha}\mathcal{G}\left(\frac{x}{m}\right)}{2^{\alpha}} \mathcal{W}(x)\nonumber\\&\leq \bigg(\frac{1}{6}\bigg)^{\alpha}{ }_{\nu} \mathcal{I}_{\mu}^{(\alpha)}	\mathcal{W}(x)\bigg[\mathcal{G}(\nu)+\mathcal{G}(\mu)+ m^{\alpha}\left(\mathcal{G}\left(\frac{\nu}{m}\right)+\mathcal{G}\left(\frac{\mu}{m}\right)+\mathcal{G}\left(\frac{\nu}{m^{2}}\right)+\mathcal{G}\left(\frac{\mu}{m^{2}}\right)\right)\bigg]\nonumber\\&\times\left[h\left(\frac{\mu-x}{\mu-\nu}\right)+h\left(\frac{x-\nu}{\mu-\nu}\right)\right]\label{feje}
	\end{flalign}
	\end{footnotesize}

holds.
\end{theorem}

\begin{proof}
Using the generalized $(h-m)$ convexity of $\mathcal{G}$ and the symmetry of $\mathcal{W}$, we obtain
	\begin{footnotesize}
\begin{flalign*}
 \frac{\mathcal{G}\left(\frac{\nu+\mu}{2}\right) { }_{\nu} \mathcal{I}_{\mu}^{(\alpha)} \mathcal{W}(x)}{2^{\alpha} h\left(\frac{1}{2}\right)}&=\frac{ { }_{\nu} \mathcal{I}_{\mu}^{(\alpha)} \mathcal{G}\left(\frac{\nu+\mu-x}{2}+\frac{m}{2}+\frac{x}{m}\right) \mathcal{W}(x)}{2^{\alpha}h\left(\frac{1}{2}\right)}\nonumber\\&\leq \frac{ { }_{\nu} \mathcal{I}_{\mu}^{(\alpha)}h\left(\frac{1}{2}\right) [\mathcal{G}(\nu+\mu-x)+{m}^{\alpha}\mathcal{G}(\frac{x}{m})]\mathcal{W}(x)}{2^{\alpha}h\left(\frac{1}{2}\right)}\nonumber\\&=\bigg(\frac{1}{2}\bigg)^{\alpha}{ }_{\nu} \mathcal{I}_{\mu}^{(\alpha)} \bigg[\mathcal{G}(\nu+\mu-x)+{m}^{\alpha}\mathcal{G}\bigg(\frac{x}{m}\bigg)\bigg]\mathcal{W}(x)\nonumber\\&=\bigg(\frac{1}{2}\bigg)^{\alpha}{ }_{\nu} \mathcal{I}_{\mu}^{(\alpha)} \bigg[\mathcal{G}(\nu+\mu-x)\mathcal{W}(\nu+\mu-x)+{m}^{\alpha}\mathcal{G}\bigg(\frac{x}{m}\bigg)\mathcal{W}(x)\bigg]\nonumber\\&=\bigg(\frac{1}{2}\bigg)^{\alpha}{ }_{\nu} \mathcal{I}_{\mu}^{(\alpha)} \bigg[\mathcal{G}(x)\mathcal{W}(x)+{m}^{\alpha}\mathcal{G}\bigg(\frac{x}{m}\bigg)\mathcal{W}(x)\bigg]\nonumber\\&=\bigg(\frac{1}{2}\bigg)^{\alpha}{ }_{\nu} \mathcal{I}_{\mu}^{(\alpha)} \bigg[\mathcal{G}(x)+m^{\alpha}\mathcal{G}\bigg(\frac{x}{m}\bigg)\bigg]\mathcal{W}(x),
	\end{flalign*}
\end{footnotesize}
which is the first part of inequality (\ref{feje}).

To show the second part of inequality (\ref{feje}), we have
	\begin{footnotesize}
	\begin{flalign*}
 &{ }_{\nu} \mathcal{I}_{\mu}^{(\alpha)} \frac{\mathcal{G}(x)+m^{\alpha} \mathcal{G}\left(\frac{x}{m}\right)}{2^{\alpha}} \mathcal{W}(x)=\bigg(\frac{1}{2}\bigg)^{\alpha}\bigg[\mathcal{I}_{\mu}^{\alpha} \mathcal{G}(x) \mathcal{W}(x)+\mathcal{G}(\nu+\mu-x) \mathcal{W}(\nu+\mu-x)\nonumber\\&+m^{\alpha} \mathcal{G}\bigg(\frac{x}{m}\bigg) \mathcal{W}(x)+m^{\alpha} \mathcal{G}\bigg(\frac{\nu+\mu-x}{m}\bigg) \mathcal{W}(\nu+\mu-x)\bigg]
=\left(\frac{1}{6}\right)^{\alpha}\left\{\nu \mathcal{I}_{\mu}^{(\alpha)}\left[\mathcal{G}\left(\frac{\mu-x}{\mu-\nu} \nu+m \frac{x-\nu}{\mu-\nu} \cdot\frac{\mu}{m}\right)\right.\right.\\ &\left.+\mathcal{G}\left(m \frac{\mu-x}{\mu-\nu} \frac{\nu}{m}+\frac{\mu-\nu}{\mu-\nu} \mu\right)\right] \mathcal{W}(\mu) + \nu \mathcal{I}_{\mu}^{(\alpha)}\left[h\left(\frac{x-\nu}{\mu-\nu} \nu+m \frac{\mu-x}{\mu-\nu} \frac{\mu}{m}\right)\right.\\ &\left.+\mathcal{G}\left(m \frac{x-\nu}{\mu-\nu} \frac{\nu}{m}+\frac{\mu-x}{\mu-\nu} \nu\right)\right] w(\mu) +{ }_{\nu} \mathcal{I}_{\mu}^{(\alpha)}\left[m^{\alpha} \mathcal{G}\left(\frac{\mu-x}{\mu-\nu} \frac{\nu}{m}+m \frac{x-\nu}{\mu-\nu} \frac{\mu}{m^{2}}\right)\right.\\ &\left.+m^{\alpha} \mathcal{G}\left(m \frac{\mu-x}{\mu-\nu} \frac{\nu}{m^{2}}+\frac{x-\nu}{\mu-\nu} \frac{\mu}{m}\right)\right] \mathcal{W}(x) +{ }_{\nu} \mathcal{I}_{\mu}^{(\alpha)}\left[m^{\alpha} \mathcal{G}\left(\frac{x-\nu}{\mu-\nu} \frac{\nu}{m}+m \frac{\mu-x}{\mu-\nu} \frac{\mu}{m^{2}}\right)\right.\\ &\left.\left.+m^{\alpha} \mathcal{G}\left(m \frac{x-\nu}{\mu-\nu} \frac{\nu}{m^{2}}+\frac{\mu-x}{\mu-\nu} \frac{\mu}{m}\right)\right] \mathcal{W}(x)\right\}
\\& \leq\bigg(\frac{1}{6}\bigg)^{\alpha}\bigg[{}_{\nu}\mathcal{I}_{\mu}^{(\alpha)}[h(\frac{\mu-x}{\mu-\nu}) \mathcal{G}(\nu)
+m^{\alpha}h(\frac{x-\nu}{\mu-\nu}) \mathcal{G}(\frac{\mu}{m})+m^{\alpha}h(\frac{\mu-x}{\mu-\nu}) \mathcal{G}(\frac{\nu}{m})+h(\frac{x-\nu}{\mu-\nu}) \mathcal{G}(\mu)] \mathcal{W}(x)
\\&+{ }_{\nu} \mathcal{I}_{\mu}^{(\alpha)}[h(\frac{x-\nu}{\mu-\nu}) \mathcal{G}(\nu)+m^{\alpha}h(\frac{\mu-x}{\mu-\nu})
\mathcal{G}(\frac{\mu}{m})+m^{\alpha}h(\frac{x-\nu}{\mu-\nu}) \mathcal{G}(\frac{\nu}{m})]
+(\frac{\mu-x}{\mu-\nu}) \mathcal{G}(\mu) \mathcal{W}(x)]
\\&+{ }_{\nu} \mathcal{I}_{\mu}^{(\alpha)}[m^{\alpha}\bigg(h((\frac{\mu-x}{\mu-\nu}) \mathcal{G}(\frac{\nu}{m}))
+m^{\alpha}h(\frac{x-\nu}{\mu-\nu}) \mathcal{G}(\frac{\mu}{m^{2}})+m^{\alpha}h(\frac{\mu-x}{\mu-\nu})
\mathcal{G}(\frac{\nu}{m^{2}})+h(\frac{x-\nu}{\mu-\nu}) \mathcal{G}(\frac{\mu}{m})\bigg)] \mathcal{W}(x)
\\&+{ }_{\nu} \mathcal{I}_{\mu}^{(\alpha)}\bigg[m^{\alpha}\bigg((h(\frac{x-\nu}{\mu-\nu}) \mathcal{G}(\frac{\nu}{m}))
+m^{\alpha}h(\frac{\mu-x}{\mu-\nu}) \mathcal{G}(\frac{\mu}{m^{2}})+m^{\alpha}h(\frac{x-\nu}{\mu-\nu})
\mathcal{G}(\frac{\nu}{m^{2}})+h(\frac{\mu-x}{\mu-\nu})\mathcal{G}(\frac{\mu}{m})\bigg) \mathcal{W}(x)\bigg]\\&=\bigg(\frac{1}{6}\bigg)^{\alpha}{ }_{\nu} \mathcal{I}_{\mu}^{(\alpha)}	\mathcal{W}(x)\bigg[\mathcal{G}(\nu)+\mathcal{G}(\mu)+ m^{\alpha}\left(\mathcal{G}\left(\frac{\nu}{m}\right)+\mathcal{G}\left(\frac{\mu}{m}\right)+\mathcal{G}\left(\frac{\nu}{m^{2}}\right)+\mathcal{G}\left(\frac{\mu}{m^{2}}\right)\right)\bigg]\\&\times\left[h\left(\frac{\mu-x}{\mu-\nu}\right)+h\left(\frac{x-\nu}{\mu-\nu}\right)\right].
	\end{flalign*}
\end{footnotesize}
\end{proof}

\begin{remark}
 Choosing \(m=1\) in Theorem \ref{fejeresult}, we have
		Theorem \ref{luo} established by Luo et al. \cite{Luo(2020)}.

\end{remark}

Using the similar technique for the proof of Theorem 4.1 given in \cite{Luo(2020)}, we studied the local Fej\'er-H-H type inequality for the differentiable generalized $(h-m)$-convex as follows.
\begin{theorem}\label{luqpartthe}
	 Suppose that \(\mathcal{G}: I \subseteq \mathbb{R} \rightarrow \mathbb{R}^{\alpha}\) is local continuous on \(I^{\circ}\). Let \(\mathcal{W}:[\nu, \mu] \rightarrow \mathbb{R}^{\alpha}, \mathcal{W} \geq 0^{\alpha}\) be a  
	 symmetric to \(\frac{\nu+\mu}{2}\) and local continuous. For \(q \geq 1\), if the mapping
	\(\left|\mathcal{G}^{(\alpha)}\right|^{q}\)  is generalized $(h-m)$-convex on \([\nu, \mu],\) then we have

	\begin{flalign}
&\bigg|\frac{\mathcal{G}(\nu)+\mathcal{G}(\mu)}{2^{\alpha}} { }_\nu \mathcal{I}_{\mu}^{(\alpha)} \mathcal{W}(x)-{ }_{\nu} \mathcal{I}_{\mu}^{(\alpha)} \mathcal{W}(x) \mathcal{G}(x)\bigg|\nonumber \\& \leq\left(\frac{(\mu-\nu)^{2}}{4}\right)^{\alpha} \frac{\|\mathcal{W}\|_{\infty}}{\Gamma(\alpha+1)}\left(\frac{\Gamma(\alpha+1)}{\Gamma(2 \alpha+1)}\right)^{1-\frac{1}{q}} \nonumber\\& \times\left(\left[\frac{\left|\mathcal{G}^{(\alpha)}(\mu)\right|^{q}}{\Gamma(\alpha+1)} \int_{0}^{1} \gamma^{\alpha} h(\gamma)(d\gamma)^{\alpha}+\frac{\left|\mathcal{G}^{(\alpha)}\left(\frac{\nu+\mu}{2m}\right)\right|^{q}}{\Gamma(\alpha+1)} \int_{0}^{1} \gamma^{\alpha}m^{\alpha} h(1-\gamma)(d \gamma)^{\alpha}\right]^{\frac{1}{q}}\right.\nonumber \\ &   +\left[\quad\left[\frac{\left.\mathcal{G}^{(\alpha)}(\nu)\right|^{q}}{\Gamma(\alpha+1)} \int_{0}^{1} \gamma^{\alpha} h(\gamma)(d\gamma)^{\alpha}+\frac{\left|\mathcal{G}^{(\alpha)}\left(\frac{\nu+\mu}{2m}\right)\right|^{q}}{\Gamma(\alpha+1)} \int_{0}^{1} \gamma^{\alpha}m^{\alpha} h(1-\gamma)(d \gamma)^{\alpha}\right]^{\frac{1}{q}}\right),\label{luqpart}
	\end{flalign}
	
	where \(\|\mathcal{W}\|_{\infty}=\sup _{\gamma \in[\nu, \mu]} \mathcal{W}(\gamma)\).
\end{theorem}

\begin{proof}
Applying generalized H\"older's inequality, Lemma \ref{luolemma} and generalized $(h-m)$-convexity, we have
\begin{equation*}
\begin{array}{l}\left|\frac{\mathcal{G}(\nu)+\mathcal{G}(\mu)}{2^{\alpha}} {}_{\nu} \mathcal{I}_{\mu}^{(\alpha)} \mathcal{W}(x)-{\nu}_{\mu}^{(\alpha)} \mathcal{W}(x) \mathcal{G}(x)\right| \\ \leq\left(\frac{\mu-\nu}{4}\right)^{\alpha} \frac{1}{\Gamma(\alpha+1)} \int_{0}^{1}\left|\frac{1}{\Gamma(1+\alpha)} \int_{m(\gamma)}^{n(\gamma)} \mathcal{W}(x)(d x)^{\alpha}\right|\left|\left(\mathcal{G}^{(\alpha)}(n(\gamma))-\mathcal{G}^{(\alpha)}(m(\gamma))\right)\right|(d \gamma)^{\alpha} \\ \leq\left(\frac{\mu-\nu}{4}\right)^{\alpha} \frac{1}{\Gamma(\alpha+1)} \int_{0}^{1}\left|\frac{1}{\Gamma(1+\alpha)} \int_{m(\gamma)}^{n(\gamma)} \mathcal{W}(x)(d x)^{\alpha}\right|\left[\left|\mathcal{G}^{(\alpha)}(n(\gamma))\right|+\left|\mathcal{G}^{(\alpha)}(m(\gamma))\right|\right](d \gamma)^{\alpha} \\ \leq\left(\frac{(\mu-\nu)^{2}}{4}\right)^{\alpha} \frac{\|\mathcal{W}\|_{\infty}}{\Gamma(1+\alpha)} \frac{1}{\Gamma(1+\alpha)} \int_{0}^{1} \gamma^{\alpha}\left[\left|\mathcal{G}^{(\alpha)}(n(\gamma))\right|+\left|\mathcal{G}^{(\alpha)}(m(\gamma))\right|\right](d\gamma)^{\alpha} \\ \leq\left(\frac{(\mu-\nu)^{2}}{4}\right)^{\alpha} \frac{\|\mathcal{W}\|_{\infty}}{\Gamma(\alpha+1)}\left(\frac{1}{\Gamma(1+\alpha)} \int_{0}^{1} \gamma^{\alpha}(\mathrm{d} \gamma)^{\alpha}\right)^{1-\frac{1}{q}}\\  \times\left\{\left[\frac{1}{\Gamma(\alpha+1)} \int_{0}^{1}\left(\gamma^{\alpha} h(\gamma)|\mathcal{G}(\alpha)(\mu)|^{q}+\gamma^{\alpha}m^{\alpha} h(1-\gamma)\left|\mathcal{G}(\alpha)\left(\frac{\nu+\mu}{2m}\right)\right|^{q}\right)(d \gamma)^{\alpha}\right]^{\frac{1}{q}}\right. \\ \left.+\left[\frac{1}{\Gamma(\alpha+1)} \int_{0}^{1}\left(\gamma^{\alpha} h(\gamma)\left|\mathcal{G}^{(\alpha)}(\nu)\right|^{q}+\gamma^{\alpha}m^{\alpha} h(1-\gamma)\left|\mathcal{G}^{(\alpha)}\left(\frac{\nu+\mu}{2m}\right)\right|^{q}\right)(d \gamma)^{\alpha}\right]^ \frac{1}{q}\right\},\end{array}
\end{equation*}

where $\frac{1}{\Gamma(1+\alpha)} \int_{0}^{1} \gamma^{\alpha}(d \gamma)^{\alpha}=\frac{\Gamma(1+\alpha)}{\Gamma(1+2 \alpha)}$.
\end{proof}

\begin{remark}
	Choosing $m=1$ in inequality (\ref{luqpart}) of Theorem \ref{luqpartthe}, we obtain inequality (4.2) of Theorem 4.1 established by Luo et al. \cite{Luo(2020)}.
\end{remark}

\begin{corollary}\label{corolaryapplication}
	Consider Theorem \ref{luqpartthe}, we have
	\begin{itemize}
		\item [1.]If $q=1$, we have
	\begin{equation*}
	\begin{array}{l}\bigg|\frac{\mathcal{G}(\nu)+\mathcal{G}(\mu)}{2^{\alpha}} {}_{\nu} \mathcal{I}_{\mu}^{(\alpha)} \mathcal{W}(x)-{}_{\nu} \mathcal{I}_{\mu}^{(\alpha)} \mathcal{W}(x) \mathcal{G}(x)\bigg|  \\ \leq\left(\frac{(\mu-\nu)^{2}}{4}\right)^{\alpha} \frac{\|\mathcal{W}\|_{\infty}}{\Gamma(1+\alpha)} {}_{0} \mathcal{I}_{1}^{(\alpha)}\left\{\gamma^{\alpha}\left[h(\gamma)\left(\left|\mathcal{G}^{(\alpha)}(\nu)\right|+\left|\mathcal{G}^{(\alpha)}(\mu)\right|\right)\right.\right. \\ \left.\left.\quad+2^{\alpha}m^{\alpha} h(1-\gamma)\left|\mathcal{G}^{(\alpha)}\left(\frac{\nu+\mu}{2m}\right)\right|\right]\right\}.\end{array}
	\end{equation*}

	\item [2.]If $h(\gamma)=\gamma^{\alpha}$ and $q=1$, we get	
	\begin{equation*}
	\begin{array}{l}\bigg|\frac{\mathcal{G}(\nu)+\mathcal{G}(\mu)}{2^{\alpha}} {}_{\nu} \mathcal{I}_{\mu}^{(\alpha)} \mathcal{W}(x)-{}_{\nu} \mathcal{I}_{\mu}^{(\alpha)} \mathcal{W}(x) \mathcal{G}(x)\bigg|  \\ \leq\left(\frac{(\mu-\nu)^{2}}{4}\right)^{\alpha} \frac{\|\mathcal{W}\|_{\infty}}{\Gamma(1+\alpha)}\left\{\frac{\Gamma(1+2 \alpha)}{\Gamma(1+3 \alpha)}\left[\left|\mathcal{G}^{(\alpha)}(\nu)\right|+\left|\mathcal{G}^{(\alpha)}(\mu)\right|\right]\right. \\ \left.\quad+2^{\alpha}m^{\alpha}\left(\frac{\Gamma(1+\alpha)}{\Gamma(1+2 \alpha)}-\frac{\Gamma(1+2 \alpha)}{\Gamma(1+3 \alpha)}\right)\left|\mathcal{G}^{(\alpha)}\left(\frac{\nu+\mu}{2m}\right)\right|\right\}.\end{array}
	\end{equation*}
	
		\item [3.]If $q=1$ and $h(\gamma)=\gamma^{\alpha s}$ where $s\in(0,1]$, we get
		\begin{equation*}
		\begin{array}{l}\bigg|\frac{\mathcal{G}(\nu)+\mathcal{G}(\mu)}{2^{\alpha}} {}_{\nu} \mathcal{I}_{\mu}^{(\alpha)} \mathcal{W}(x)-{}_{\nu} \mathcal{I}_{\mu}^{(\alpha)} \mathcal{W}(x) \mathcal{G}(x)\bigg|  \\ \leq\left(\frac{(\mu-\nu)^{2}}{4}\right)^{\alpha} \frac{\|\mathcal{W}\|_{\infty}}{\Gamma(1+\alpha)}\left(\frac{\Gamma(1+(s+1) \alpha)}{\Gamma(1+(s+2) \alpha)}\left[\left|\mathcal{G}^{(\alpha)}(\nu)\right|+\left|\mathcal{G}^{(\alpha)}(\mu)\right|\right]\right. \\ \left.\quad+2^{\alpha}m^{\alpha} B_{\alpha}(2, s+1)\left|\mathcal{G}(\alpha)\left(\frac{\nu+\mu}{2m}\right)\right|\right),\end{array}
		\end{equation*}

		where $B_{\alpha}(\nu, \mu)=\int_{0}^{1} \gamma^{(\nu-1) \alpha}(1-\gamma)^{(\mu-1) \alpha}(d \gamma)^{\alpha}$, for any $\nu, \mu \in [0,\infty)$.
		
	\end{itemize}	
\end{corollary}

\section{Applications}
\subsection{Random variables}
Suppose that \(\mathrm{X}\) is a random variable. Let \(p:[\nu, \mu] \rightarrow \mathbb{R}^{\alpha}\) be the generalized probability distribution mapping for all \(\gamma \in[\nu, \mu]\). The function $p$ possesses the upper and lower bounds that is \(\alpha\)-type real numbers \(\varPsi, \varOmega\) with
\(0^{\alpha} \leq \varPsi \leq p(\gamma) \leq \varOmega \leq 1^{\alpha} .\) The generalized expectation and \(r\)-moment are respectively given as \cite{Erden(2016)}
\begin{equation*}
E^{\alpha}(X)=\frac{1}{\Gamma(1+\alpha)} \int_{\nu}^{\mu} \gamma^{\alpha} p(\gamma)(d \gamma)^{\alpha}
\end{equation*}
and
\begin{equation*}
E_{r}^{\alpha}(X)=\frac{1}{\Gamma(1+\alpha)} \int_{\nu}^{\mu} \gamma^{r \alpha} p(\gamma)(d \gamma)^{\alpha}, r \geq 0.
\end{equation*}

Suppose that \(\mathcal{W}(x) \in C_{\alpha}[\nu, \mu]\) is the generalized probability density mapping of \(X\) which is symmetric to \(\frac{\nu+\mu}{2}\), for $0<\nu<\mu$. If \(\mathcal{G}(x)=\)
\(x^{r \alpha}\) with $r \geq 1$ we obtain the result which is related to $r$-moment.

\begin{proposition}
	If we choose \(h(\gamma)=\gamma^{\alpha}\) in Corollary (\ref{corolaryapplication}), we get
	\begin{equation*}
	\begin{array}{l}\left|\frac{\nu^{r \alpha}+\mu^{r \alpha}}{2^{\alpha}}{ }_{\nu} \mathcal{I}_{\mu}^{(\alpha)} \mathcal{W}(x)-E_{r}^{\alpha}(X)\right| \\ \leq\left(\frac{(\mu-\nu)^{2}}{4}\right)^{\alpha} \frac{\Gamma(r \alpha+1)}{\Gamma((r-1)\alpha+1 )} \frac{\|\mathcal{W}\|_{\infty}}{\Gamma(1+\alpha)} \\ \quad \times\left(\frac{\Gamma(2 \alpha+1)}{\Gamma(3 \alpha+1)}\left[\nu^{\alpha(r-1) }+\mu^{\alpha(r-1) }\right]\right. \\ \left.\quad+2^{\alpha}m^{\alpha}\left(-\frac{\Gamma(2 \alpha+1)}{\Gamma(1+3 \alpha)}+\frac{\Gamma(1+\alpha)}{\Gamma(2 \alpha+1)}\right)\left(\frac{\nu+\mu}{2m}\right)^{(r-1) \alpha}\right).\end{array}
	\end{equation*}
\end{proposition}

\subsection{Numerical integration}

Let \(X_{i}: \nu=x_{0}<x_{1}<\ldots<x_{i-1}<x_{i}=\mu\) be a partition of the
interval \([\nu, \mu], \varepsilon_{j} \in\left[x_{j}, x_{j+1}\right](j=0, \ldots, i-1) .\) We consider
the following trapezoidal quadrature formula

\begin{equation*}
\frac{1}{\Gamma(1+\alpha)} \int_{\nu}^{\mu} \mathcal{W}(\gamma) \mathcal{G}(\gamma)(d \gamma)^{\alpha}=T(\mathcal{G}, \mathcal{W}, \varepsilon)+R_{T}(\mathcal{G}, \mathcal{W}, \varepsilon),
\end{equation*}
where
\begin{equation*}
T(\mathcal{G}, \mathcal{W},\varepsilon):=\frac{1}{\Gamma(1+\alpha)} \sum_{j=0}^{n-1} \frac{\mathcal{G}\left(x_{j}\right)+\mathcal{G}\left(x_{j+1}\right)}{2^{\alpha}} \int_{x_{j}}^{x_{j+1}} \mathcal{W}(\gamma)(d \gamma)^{\alpha}.
\end{equation*}
Here  \(R_{T}(\mathcal{G}, \mathcal{W}, \varepsilon)\) is the related approximation
error of \(\frac{1}{\Gamma(1+\alpha)} \int_{\nu}^{\mu} \mathcal{G}(\gamma) \mathcal{W}(\gamma)(d \gamma)^{\alpha}.\) Therefore, we obtain the following result.

\begin{proposition}
	 Consider the assumptions of Theorem \ref{luqpartthe}. The weighted second part of Fej\'er-H-
	H error estimate is given as follows:
	
	\begin{equation}
	\begin{array}{l}\bigg|R_{M}(\mathcal{G}, \mathcal{W}, \varepsilon)\bigg| \\ \leq \sum_{j=0}^{i-1}\left(\frac{\left(x_{j+1}-x_{j}\right)^{2}}{4}\right)^{\alpha} \frac{\|\mathcal{W}\|_{\infty}}{\Gamma(1+\alpha)}\left\{\frac{\Gamma(1+2 \alpha)}{\Gamma(1+3 \alpha)}\right. \\ \quad \times\left[\left|\mathcal{G}^{(\alpha)}\left(x_{j}\right)\right|+\left|\mathcal{G}^{(\alpha)}\left(x_{j+1}\right)\right|\right]+2^{\alpha}m^{\alpha}\left(\frac{\Gamma(1+\alpha)}{\Gamma(1+2 \alpha)}-\frac{\Gamma(1+2 \alpha)}{\Gamma(1+3 \alpha)}\right) \\ \left.\quad \times\left|\mathcal{G}^{(\alpha)}\left(\frac{x_{j+1}+x_{j}}{2m}\right)\right|\right\}.\label{applicationnumurical}\end{array}
	\end{equation}
\end{proposition}

\begin{proof}
Applying Theorem \ref{luqpartthe}, $h(\gamma)=\gamma^{\alpha}$ and \(q=1\) on the interval
	\(\left[x_{j}, x_{j+1}\right],\) we have
	\begin{equation*}
	\begin{array}{l}\bigg| \frac{1}{\Gamma(1+\alpha)} \frac{\mathcal{G}\left(x_{j}\right)+\mathcal{G}\left(x_{j+1}\right)}{2^{\alpha}} \int_{x_{j}}^{x_{j+1}} \mathcal{W}(\gamma)(d \gamma)^{\alpha} \\ \quad-\frac{1}{\Gamma(1+\alpha)} \int_{x_{j}}^{x_{j+1}} \mathcal{W}(\gamma) \mathcal{G}(\gamma)(d \gamma)^{\alpha} \bigg| \\ \leq\left(\frac{\left(x_{j+1}-x_{j}\right)^{2}}{4}\right)^{\alpha} \frac{\|\mathcal{W}\|_{\infty}}{\Gamma(1+\alpha)}\left\{\frac{\Gamma(1+2 \alpha)}{\Gamma(1+3 \alpha)}\left[\left|\mathcal{G}^{(\alpha)}\left(x_{j}\right)\right|+\left|\mathcal{G}^{(\alpha)}\left(x_{j+1}\right)\right|\right]\right. \\ \left.\quad+2^{\alpha}m^{\alpha}\left(-\frac{\Gamma(2 \alpha+1)}{\Gamma(3 \alpha+1)}+\frac{\Gamma(\alpha+1)}{\Gamma(2 \alpha+1)}\right)\left|\mathcal{G}^{(\alpha)}\left(\frac{x_{j+1}+x_{j}}{2m}\right)\right|\right\}\end{array}
	\end{equation*}
	for all \(j=0, \ldots, i-1 .\) Summing over \(j\) from 0 to \(i-1,\) we have
	the inequality (\ref{applicationnumurical}).
\end{proof}


\begin{thebibliography}{999}
	\bibitem{Yang(2012)}Yang, X. J. Advanced local fractional calculus and its applications. New York: WorldSicince; 2012.
	\bibitem{Erden(2016)}Erden, S. A. M. E. T., Sarikaya, M. Z., \& Çelik, N. U. R. I. (2016). Some generalized inequalities involving local fractional integrals and their applications for random variables and numerical integration. Journal of Applied Mathematics, Statistics and Informatics, 12(2), 49-65.
	\bibitem{Dragomir(2003)}Dragomir, S. S., \& Pearce, C. (2003). Selected topics on Hermite-Hadamard inequalities and applications. Mathematics Preprint Archive, 2003(3), 463-817.
	\bibitem{Feje} L. Fej\'er, \"Uber die Fourierreihen, II, Math. Naturwiss. Anz Ungar. Akad. Wiss. 24(1906) 369-390 (in Hungarian).
	\bibitem{Sarikaya(2008)}Sarikaya, M. Z., Saglam, A., \& Yildirim, H. (2008). On some Hadamard-type inequalities for h-convex functions. J. Math. Inequal, 2(3), 335-341.
	\bibitem{Dragomir(2002)}Dragomir, S. S. (2002). On some new inequalities of Hermite-Hadamard type for $ m $-convex functions. Tamkang Journal of Mathematics, 33(1), 45-56.
		\bibitem{M(2020)}Mehmood, S.,\& Farid, G. (2020). Fractional integrals inequalities for exponentially m-convex functions. Open Journal of Mathematical Sciences, 4(1), 78-85.
	\bibitem{Toader(1985)}G. Toader, Some generalization of the convexity, Proc. Colloq. Approx. Opt., Cluj-Napoca,
	(1985), 329-338.
	\bibitem{Sarikaya2013}Sarikaya, M. Z., Set, E., Yaldiz, H., \& Başak, N. (2013). Hermite–Hadamard’s inequalities for fractional integrals and related fractional inequalities. Mathematical and Computer Modelling, 57(9-10), 2403-2407.
		\bibitem{K2020}Kumar, S.,\& Atangana, A. (2020).. A numerical study of the nonlinear fractional mathematical model of tumor cells in presence of chemotherapeutic treatment. International Journal of Biomathematics, 13(03), 2050021.
			\bibitem{Da2020}Danane, J., Allali, K., \& Hammouch, Z. (2020). Mathematical analysis of a fractional differential model of HBV infection with antibody immune response. Chaos, Solitons \& Fractals, 136, 109787.
		\bibitem{A2020}Almutairi, O., \& Kılı\c{c}man, A. (2020). Integral inequalities for s-convexity via generalized fractional integrals on fractal sets. Mathematics, 8(1), 53.
			\bibitem{Dragomir2020}Dragomir, S. S. (2020). Some Hermite–Hadamard type integral inequalities for convex functions defined on convex bodies in $\mathbb{R}^{n}$. Journal of Applied Analysis, 26(1), 67-77.
			\bibitem{Qi2020}Qi, H. X., Yussouf, M., Mehmood, S., Chu, Y. M., \& Farid, G. (2020). Fractional integral versions of Hermite-Hadamard type inequality for generalized exponentially convexity. AIMS Math, 5(6), 6030-6042.
			\bibitem{setb2020}Set, E., Butt, S. I., Akdemir, A. O., Karaoǧlan, A., \& Abdeljawad, T. New integral inequalities for differentiable convex functions via Atangana-Baleanu fractional integral operators. Chaos, Solitons \& Fractals, 143, 110554.
			\bibitem{Atangana2020}Atangana, A.(2020). Modelling the spread of COVID-19 with new fractal-fractional operators:Can the lockdown save mankind before vaccination? Chaos, Solitons \& Fractals, 143, 110554.
			\bibitem{Al2020}Almutairi, O., \& Kılı\c{c}man, A. (2020). New Generalized Hermite-Hadamard Inequality and Related Integral Inequalities Involving Katugampola Type Fractional Integrals. Symmetry, 12(4), 568.
		\bibitem{A2019}Almutairi, O., \& Kılı\c{c}man, A. (2019). New fractional inequalities of midpoint type via s-convexity and their application. Journal of Inequalities and Applications, 2019(1), 1-19.
		\bibitem{Ozcan(2019)}Özcan, S., \& I\c{s}can, İ. (2019). Some new Hermite–Hadamard type inequalities for s-convex functions and their applications. Journal of Inequalities and Applications, 2019(1), 201.
			\bibitem{Almutairi2019}Almutairi, O., \& Kılı\c{c}man, A. (2019).  Generalized Integral Inequalities for Hermite–Hadamard-Type Inequalities via s-Convexity on Fractal Sets. Mathematics, 7(11), 1065.
		\bibitem{Ozcan(2020)}\"Ozcan, S. (2020).–Hadamard type inequalities for m-convex and $(\alpha, m) $-convex functions. Journal of Inequalities and Applications, 2020(1), 1-10.
	\bibitem{Moetal(2014)}Mo, H., \& Sui, X. (2014). Generalized-convex functions on fractal sets. In Abstract and Applied Analysis (Vol. 2014). Hindawi.
	\bibitem{Luo(2020)}Luo, C., Wang, H., \& Du, T. (2020). Fejér–Hermite–Hadamard type inequalities involving generalized h-convexity on fractal sets and their applications. Chaos, Solitons \& Fractals, 131, 109547.
		\bibitem{Ozdemir(2016)}\"Ozdemir, M. E., Akdemir, A. O., \& Set, E. (2016). On (hm)-convexity and Hadamard-type inequalities. Transylv.J. Math. Mech. 8(1), 51–58.
			\bibitem{Am(2019)}Almutairi, O., \& Kılı\c{c}man, A. (2019). Some integral inequalities for h-Godunova-Levin preinvexity. Symmetry, 11(12), 1500.
				\bibitem{Du(2019)}Du, T., Wang, H., Khan, M. A., \& Zhang, Y. (2019). Certain integral inequalities considering generalized m-convexity on fractal sets and their applications. Fractals, 27(07), 1950117.
			\bibitem{DA(1998)}Dragomir, S. S., \& Agarwal, R. (1998). Two inequalities for differentiable mappings and applications to special means of real numbers and to trapezoidal formula. Applied Mathematics Letters, 
				\bibitem{Dragomir(1999)}Dragomir, S. S., \& Fitzpatrick, S. (1999). The Hadamard inequalities for s-convex functions in the second sense. Demonstratio Mathematica, 32(4), 687-696.
						\bibitem{Vivas(2016)}Vivas, M., Hernández, J., \& Merentes, N. (2016). New Hermite-Hadamard and Jensen type inequalities for h-convex functions on fractal sets. Revista Colombiana de Matemáticas, 50(2), 145-164.
										\bibitem{Varosanec(2007)}Varosanec S. On h-convexity. J Math Anal Appl 2007;326:303-11.
										\bibitem{Anastassiou(2019)}Anastassiou, G., Kashuri, A., \& Liko, R. (2019). Local fractional integrals involving generalized strongly m-convex mappings. Arabian Journal of Mathematics, 8(2), 95-107.
										\bibitem{set(2021)}Set, E., Butt, S. I., Akdemir, A. O., Karaoǧlan, A., \& Abdeljawad, T. New integral inequalities for differentiable convex functions via Atangana-Baleanu fractional integral operators. Chaos, Solitons \& Fractals, 143, 110554..
	\bibitem{Abdeljawad(2020)}Abdeljawad, T., Rashid, S., Hammouch, Z., \& Chu, Y. M. (2020). Some new local fractional inequalities associated with generalized (s, m) $(s, m) $-convex functions and applications. Advances in Difference Equations, 2020(1), 1-27.
\end{thebibliography}
\end{document}